\documentclass[12pt]{amsart}
\usepackage{mathrsfs}
\usepackage{amsfonts}
\usepackage{amssymb,amsfonts}

\begin{document}
\theoremstyle{plain}
\newtheorem{thm}{Theorem}[section]
\newtheorem{theorem}[thm]{Theorem}
\newtheorem{lemma}[thm]{Lemma}
\newtheorem{corollary}[thm]{Corollary}
\newtheorem{proposition}[thm]{Proposition}
\newtheorem{conjecture}[thm]{Conjecture}
\newtheorem{obs}[thm]{}
\theoremstyle{definition}
\newtheorem{construction}[thm]{Construction}
\newtheorem{notations}[thm]{Notations}
\newtheorem{question}[thm]{Question}
\newtheorem{problem}[thm]{Problem}
\newtheorem{remark}[thm]{Remark}
\newtheorem{remarks}[thm]{Remarks}
\newtheorem{definition}[thm]{Definition}
\newtheorem{claim}[thm]{Claim}
\newtheorem{assumption}[thm]{Assumption}
\newtheorem{assumptions}[thm]{Assumptions}
\newtheorem{properties}[thm]{Properties}
\newtheorem{example}[thm]{Example}
\newtheorem{comments}[thm]{Comments}
\newtheorem{blank}[thm]{}
\newtheorem{defn-thm}[thm]{Definition-Theorem}

\newcommand{\sM}{{\mathcal M}}


\title[An effective recursion formula]{An effective recursion formula for computing intersection numbers}
        \author{Kefeng Liu}
        \address{Center of Mathematical Sciences, Zhejiang University, Hangzhou, Zhejiang 310027, China;
                Department of Mathematics,University of California at Los Angeles,
                Los Angeles, CA 90095-1555, USA}
        \email{liu@math.ucla.edu, liu@cms.zju.edu.cn}
        \author{Hao Xu}
        \address{Center of Mathematical Sciences, Zhejiang University, Hangzhou, Zhejiang 310027, China}
        \email{haoxu@cms.zju.edu.cn}

        \begin{abstract}
           We prove a new effective recursion formula for computing all intersection indices (integrals of $\psi$ classes) on the moduli space of
           curves, inducting only on the genus.
        \end{abstract}
    \maketitle

\maketitle

\section{Introduction}

We denote by $\overline{\sM}_{g,n}$ the moduli space of stable
$n$-pointed genus $g$ complex algebraic curves. Let $\psi_i$ be the
first Chern class of the line bundle whose fiber over each pointed
stable curve is the cotangent line at the $i$-th marked point.

We adopt Witten's notation in this paper,
$$\langle\tau_{d_1}\cdots\tau_{d_n}\rangle_g:=\int_{\overline{\sM}_{g,n}}\psi_1^{d_1}\cdots\psi_n^{d_n}.$$

Witten-Kontsevich theorem \cite{Wi, Ko} provides a recursive way to
compute all these intersection numbers. However explicit and
effective recursion formulae for computing intersection indices are
still very rare and very welcome.

Our $n$-point function formula \cite{LX} computes intersection
indices recursively by decreasing the number of marked points. So it
is natural to ask whether there exists a recursion formula which
explicitly expresses intersection indices in terms of intersection
indices with strictly lower genus. Motivated by Witten's KdV
coefficient equation and our $n$-point function formula, we find
such a recursion formula.
\begin{theorem}
Let $d_j\geq0$ and $\sum_{j=1}^n d_j=3g+n-3$. Then
\begin{multline*}
(2g+n-1)(2g+n-2)\langle\prod_{j=1}^n\tau_{d_j}\rangle_g\\
=\frac{2d_1+3}{12}\langle\tau_0^4\tau_{d_1+1}\prod_{j=2}^n\tau_{d_j}\rangle_{g-1}-\frac{2g+n-1}{6}\langle\tau_0^3\prod_{j=1}^n\tau_{d_j}\rangle_{g-1}\\
+\sum_{\{2,\dots,n\}=I\coprod
J}(2d_1+3)\langle\tau_{d_1+1}\tau_0^2\prod_{i\in
I}\tau_{d_i}\rangle_{g'}\langle\tau_0^2\prod_{i\in
J}\tau_{d_i}\rangle_{g-g'}\\
-\sum_{\{2,\dots,n\}=I\coprod
J}(2g+n-1)\langle\tau_{d_1}\tau_0\prod_{i\in
I}\tau_{d_i}\rangle_{g'}\langle\tau_0^2\prod_{i\in
J}\tau_{d_i}\rangle_{g-g'}.
\end{multline*}
It's not difficult to see that when indices $d_j\geq1$, all non-zero
intesection indices on the right hands have genera strictly less
than $g$.
\end{theorem}

It is not difficult to see that the above recursion formula,
together with the string and dilaton equations, provides an
effective recursive algorithm for computing intersection indices on
moduli spaces of curves by inducting solely on genus $g$. We have
written a Maple program implementing the above recursion formula to
compute intersection indices which is available at \cite{Maple}.

Besides our $n$-point function formula and the above new recursion
formula, the only known effective formula for computing intersection
indices is the following DVV formula \cite{Dij, DVV} (equivalent to
Virasoro constraints)
\begin{multline*}
\langle\tau_{k+1}\tau_{d_1}\cdots\tau_{d_n}\rangle_g=\frac{1}{(2k+3)!!}\left[\sum_{j=1}^n
\frac{(2k+2d_j+1)!!}{(2d_j-1)!!}\langle\tau_{d_1}\cdots
\tau_{d_{j}+k}\cdots\tau_{d_n}\rangle_g\right.\\
+\frac{1}{2}\sum_{r+s=k-1}
(2r+1)!!(2s+1)!!\langle\tau_r\tau_s\tau_{d_1}\cdots\tau_{d_n}\rangle_{g-1}\\
\left.+\frac{1}{2}\sum_{r+s=k-1} (2r+1)!!(2s+1)!!
\sum_{\underline{n}=I\coprod J}\langle\tau_r\prod_{i\in
I}\tau_{d_i}\rangle_{g'}\langle\tau_s\prod_{i\in
J}\tau_{d_i}\rangle_{g-g'}\right]
\end{multline*}
which computes intersection indices by inducting on both the genus
and the number of marked points. We know that Mirzakhani's recursion
formula \cite{Mir} of Weil-Petersson volumes is essentially
equivalent to the DVV formula \cite{LX2, MS}.

We also found the following simple identity, which plays a key role
in the proof of Theorem 1.1.
\begin{theorem}
Let $d_j\geq 0$ and $\sum_{j=1}^n d_j=3g+n-2$. Then
$$(2g+n-1)\langle\tau_0\prod_{j=1}^n\tau_{d_j}\rangle_g=\frac{1}{12}\langle\tau_0^4\prod_{j=1}^n\tau_{d_j}\rangle_{g-1}
+\frac{1}{2}\sum_{\underline{n}=I\coprod
J}\langle\tau_0^2\prod_{i\in
I}\tau_{d_i}\rangle_{g'}\langle\tau_0^2\prod_{i\in
J}\tau_{d_i}\rangle_{g-g'}.$$
\end{theorem}
In fact, Theorem 1.2 is also very suitable for computing
intersection indices. Note that the non-zero intersection indices on
the right hand side have strictly lower genus. We may then compute
inductively on the maximum index, say $d_1$, and use the string
equation
$$\langle\tau_{d_1}\prod_{j=2}^n\tau_{d_j}\rangle_g=
\langle\tau_0\tau_{d_1+1}\prod_{j=2}^n\tau_{d_j}\rangle_g-\sum_{i=2}^n\langle\tau_{d_1+1}\tau_{d_i-1}\prod_{j\neq
i,1}\tau_{d_j}\rangle_g.$$

Both Theorem 1.1 and 1.2 are proved by applying our $n$-point
function formula and Witten's KdV coefficient equation.

Theorem 1.1 tells us that the intersection numbers on moduli spaces
of curves are determined by intersection numbers on the boundaries.
On the other hand, a theorem of Ionel \cite{Ion} says when $g\geq
2$, any product of degree at least $g$ of descendant or tautological
classes vanishes when restricted to $\sM_{g,n}$.

\vskip 30pt
\section{Proof of Theorems 1.1 and 1.2}

\begin{definition}
We call the following generating function
$$F(x_1,\dots,x_n)=\sum_{g=0}^{\infty}\sum_{\sum d_j=3g-3+n}\langle\tau_{d_1}\cdots\tau_{d_n}\rangle_g\prod_{j=1}^n x_j^{d_j}$$
the $n$-point function.
\end{definition}

Consider the following ``normalized'' $n$-point function
$$G(x_1,\dots,x_n)=\exp\left(\frac{-\sum_{j=1}^n
x_j^3}{24}\right)\cdot F(x_1,\dots,x_n).$$ We have the following
simple recursion formula of $n$-point functions.
\begin{theorem} \cite{LX} For $n\geq2$,
\begin{equation*}
G(x_1,\dots,x_n)=\sum_{r,s\geq0}\frac{(2r+n-3)!!}{4^s(2r+2s+n-1)!!}P_r(x_1,\dots,x_n)\Delta(x_1,\dots,x_n)^s,
\end{equation*}
where $P_r$ and $\Delta$ are homogeneous symmetric polynomials
defined by
\begin{align*}
\Delta(x_1,\dots,x_n)&=\frac{(\sum_{j=1}^nx_j)^3-\sum_{j=1}^nx_j^3}{3},\nonumber\\
P_r(x_1,\dots,x_n)&=\left(\frac{1}{2\sum_{j=1}^nx_j}\sum_{\underline{n}=I\coprod
J}(\sum_{i\in I}x_i)^2 (\sum_{i\in J}x_i)^2 G(x_I)
G(x_J)\right)_{3r+n-3}\nonumber\\
&=\frac{1}{2\sum_{j=1}^nx_j}\sum_{\underline{n}=I\coprod
J}(\sum_{i\in I}x_i)^2(\sum_{i\in J}x_i)^2\sum_{r'=0}^r
G_{r'}(x_I)G_{r-r'}(x_J),
\end{align*}
where $I,J\ne\emptyset$, $\underline{n}=\{1,2,\ldots,n\}$ and
$G_g(x_I)$ denotes the degree $3g+|I|-3$ homogeneous component of
the normalized $|I|$-point function $G(x_{k_1},\dots,x_{k_{|I|}})$,
where $k_j\in I$.
\end{theorem}

We also need the following lemma.
\begin{lemma} \cite{LX} Let $P$ and $\Delta$ be as defined in Theorem
2.2. Then
$$G_g(x_1,\dots,x_n)=
\frac{1}{(2g+n-1)}P_g(x_1,\dots,x_n)+\frac{\Delta(x_1,\dots,x_n)}{4(2g+n-1)}G_{g-1}(x_1,\dots,x_n).$$
\end{lemma}

In terms of $n$-point functions, it's not difficult to see that
Theorem 1.2 can be rephrased as the following proposition.
\begin{proposition} Let $F(x_1,\dots,x_n)$ be $n$-point functions.
Then
\begin{multline*}
\sum_{g=0}^\infty (2g+n-1)\left(\sum_{i=1}^n x_i\right)F_g(x_1,\dots,x_n)=\frac{1}{12}\left(\sum_{i=1}^n x_i\right)^4F(x_1,\dots,x_n)\\
+\frac{1}{2}\sum_{\underline{n}=I\coprod J} \left(\sum_{i\in I}
x_i\right)^2\left(\sum_{i\in J} x_i\right)^2 F(x_I)F(x_J).
\end{multline*}
\end{proposition}
\begin{proof}
For convenience of notation, we define
$$H=\exp\left(\frac{\sum_{i=1}^n
x_i^3}{24}\right),\qquad H^{-1}=\exp\left(\frac{-\sum_{i=1}^n
x_i^3}{24}\right),$$
$$H_d=\frac{1}{d!}\left(\frac{\sum_{i=1}^n x_i^3}{24}\right)^d,\qquad H^{-1}_d=\frac{1}{d!}\left(\frac{-\sum_{i=1}^n x_i^3}{24}\right)^d.$$
Note that $\sum_{i=1}^d H_iH^{-1}_{d-i}=0$ if $d>0$. We have
\begin{multline*}
\frac{H^{-1}\cdot RHS}{\sum_{i=1}^n x_i}=\sum_{g=0}^\infty\left(\frac{1}{12}\left(\sum_{i=1}^n x_i\right)^3 G_{g-1}(x_1,\dots,x_n)+P_g(x_1,\dots,x_n)\right)\\
=\sum_{g=0}^\infty
\left((2g+n-1)G_g(x_1,\dots,x_n)+\frac{1}{12}\left(\sum_{i=1}^n
x_i^3\right) G_{g-1}(x_1,\dots,x_n)\right)
\end{multline*}
where we applied Lemma 2.3 in the second equation.

\begin{multline*}
\frac{H^{-1}\cdot LHS}{\sum_{i=1}^n x_i}=\sum_{g=0}^\infty\sum_{a+b+c=g}(2a+2b+n-1)G_a(x_1,\dots,x_n)H_b H^{-1}_c\\
=\sum_{g=0}^\infty\sum_{a=0}^g(2a+n-1)G_a(x_1,\dots,x_n)\sum_{b+c=g-a}H_b
H^{-1}_c\\
+\sum_{g=0}^\infty\sum_{a+b+c=g}G_a(x_1,\dots,x_n)2b H_b H^{-1}_c\\
=\sum_{g=0}^\infty(2g+n-1)G_g(x_1,\dots,x_n)+\sum_{g=0}^\infty\frac{1}{12}\left(\sum_{i=1}^n
x_i^3\right) G_{g-1}(x_1,\dots,x_n).
\end{multline*}
So we conclude the proof of the proposition.
\end{proof}

The following is a reformulation of Witten's KdV coefficient
equation (see \cite{LX}).
\begin{lemma} We have
\begin{multline*}
\left(2y\frac{\partial}{\partial y}+1\right)
\left(\left(y+\sum_{j=1}^n x_j\right)^2 F(y,x_1,\dots,x_n)\right)=\\
\left(\frac{y}{4}\left(y+\sum_{j=1}^n x_j\right)^4+y\left(y+\sum_{j=1}^n x_j\right)\right) F(y,x_1,\dots,x_n)\\
+y\sum_{\substack{\underline{n}=I\coprod J\\J\neq\emptyset}}
\left(\left(y+\sum_{i\in I} x_i\right)\left(\sum_{i\in J}
x_i\right)^3+2\left(y+\sum_{i\in I} x_i\right)^2\left(\sum_{i\in J}
x_i\right)^2\right) F(y,x_I)F(x_J).
\end{multline*}
\end{lemma}

From Theorem 1.2, we can group the first and third terms on the
right hand side of Theorem 1.1 and further simplify to the following
recursion relation.
\begin{multline*}
(2g+n-1)\langle\tau_r\prod_{j=1}^n\tau_{d_j}\rangle_g=(2r+3)\langle\tau_0\tau_{r+1}\prod_{j=1}^n\tau_{d_j}\rangle_g\\
-\frac{1}{6}\langle\tau_0^3\tau_r\prod_{j=1}^n\tau_{d_j}\rangle_{g-1}
-\sum_{\underline{n}=I\coprod J}\langle\tau_0\tau_r\prod_{i\in
I}\tau_{d_i}\rangle_{g'}\langle\tau_0^2\prod_{i\in
J}\tau_{d_i}\rangle_{g-g'}.
\end{multline*}
So we need only prove the following equivalent statement of Theorem
1.1.
\begin{proposition} We have
\begin{multline*}
y\sum_{g=0}^\infty(2g+n-1)
F_g(y,x_1,\dots,x_n)=2y\frac{\partial}{\partial y}
\left(\left(\sum_{j=1}^n y+x_j\right) F(y,x_1,\dots,x_n)\right)\\
+\left(\left(y+\sum_{j=1}^n x_j\right)-\frac{y}{6}\left(y+\sum_{j=1}^n x_j\right)^3\right) F(y,x_1,\dots,x_n)\\
-y\sum_{\substack{\underline{n}=I\coprod J\\J\neq\emptyset}}
\left(y+\sum_{i\in I} x_i\right)\left(\sum_{i\in J} x_i\right)^2
 F(y,x_I)F(x_J).
\end{multline*}
\end{proposition}
\begin{proof}
From Lemma 2.5, it's not difficult to get the following equation for
the part of differentiation with respect to $y$.
\begin{multline*}
2y\left(y+\sum_{j=1}^n x_j\right)\frac{\partial}{\partial
y}\left(\left(\sum_{j=1}^n y+x_j\right) F(y,x_1,\dots,x_n)\right)\\
=\left(\frac{y}{4}\left(y+\sum_{j=1}^n x_j\right)^4-y\left(y+\sum_{j=1}^n x_j\right)-\left(y+\sum_{j=1}^n x_j\right)^2\right) F(y,x_1,\dots,x_n)\\
+y\sum_{\substack{\underline{n}=I\coprod J\\J\neq\emptyset}}
\left(\left(y+\sum_{i\in I} x_i\right)\left(\sum_{i\in J}
x_i\right)^3+2\left(y+\sum_{i\in I} x_i\right)^2\left(\sum_{i\in J}
x_i\right)^2\right) F(y,x_I)F(x_J).
\end{multline*}

Multiply each side of the equation in Proposition 2.6 by
$y+\sum_{j=1}^n x_j$ and substitute the differential part using the
above equation, we get
\begin{multline*}
y\sum_{g=0}^\infty(2g+n-1)\left(y+\sum_{i=1}^n x_i\right) F_g(y,x_1,\dots,x_n)\\
=\left(\frac{y}{12}\left(y+\sum_{j=1}^n x_j\right)^4-y\left(y+\sum_{j=1}^n x_j\right)\right) F(y,x_1,\dots,x_n)\\
+y\sum_{\substack{\underline{n}=I\coprod J\\J\neq\emptyset}}
\left(y+\sum_{i\in I} x_i\right)^2\left(\sum_{i\in J} x_i\right)^2
 F(y,x_I)F(x_J).
\end{multline*}
Add to each side with the term $$y\left(y+\sum_{j=1}^n x_j\right)
F(y,x_1,\dots,x_n),$$ we get the equation of Proposition 2.4. So we
conclude the proof of Theorem 1.1.
\end{proof}

$$ \ \ \ \ $$

\end{document}